\documentclass[12pt]{amsart}

\usepackage{graphicx,amsmath,amssymb,amsthm,enumerate}
\usepackage[colorlinks=true,linkcolor=black,citecolor=black]{hyperref} 
\usepackage[all]{xy}
\usepackage{enumerate}
\usepackage{mathrsfs}
\usepackage{graphicx}

\newcommand{\C}{\mathbb{C}}
\newcommand{\N}{\mathbb{N}}
\newcommand{\rank}{{\operatorname{rank}}}
\newcommand{\Ker}{{\operatorname{Ker}}}
\newcommand{\Image}{{\operatorname{Im}}}
\newcommand{\pmt}[1]{\begin{pmatrix}#1\end{pmatrix}}
\newcommand{\bmt}[1]{\begin{bmatrix}#1\end{bmatrix}}

\newtheorem{pro}{Proposition}[section]
\newtheorem{Lem}[pro]{Lemma}
\newtheorem{Theo}[pro]{Theorem}
\newtheorem*{Theoetoile}{Theorem} 
\theoremstyle{definition}
\newtheorem{Defi}[pro]{Definition}
\newtheorem{Exem}[pro]{Example}

\numberwithin{equation}{section}
\usepackage{hyperref}
\hypersetup{
	colorlinks=true,       
	linkcolor=red,          
	citecolor=blue,        
	filecolor=magenta,      
	urlcolor=cyan           
}

\title[A new proof of Stanley's theorem on the SLP]{A new proof of Stanley's theorem on the strong Lefschetz property}
\author{Ho V. N. Phuong}
\address{University of Sciences, Hue University,  77 Nguyen Hue St., Hue, Vietnam}
\email{hvnphuong@husc.edu.vn}
\author{Quang Hoa Tran}
\address{University of Education, Hue University, 34 Le Loi St., Hue City, Viet Nam}
\email{tranquanghoa@hueuni.edu.vn}

\begin{document}
\maketitle

\begin{abstract}
A standard graded artinian monomial complete intersection algebra $A=\Bbbk[x_1,x_2,\ldots,x_n]/(x_1^{a_1},x_2^{a_2},\ldots,x_n^{a_n})$, with $\Bbbk$ a field of characteristic zero, has the strong Lefschetz property due to Stanley in 1980. In this paper, we give a new proof for this result by using only the basic properties of linear algebra. Furthermore, our proof is still true in the case where the characteristic of $\Bbbk$ is  greater than the socle degree of $A$, namely $a_1+a_2+\cdots+a_n - n$.
\end{abstract}

\section{Introduction}
Let $\Bbbk$ be a field and $R=\Bbbk[x_1,x_2,\ldots,x_n]$ be a standard graded polynomial ring over $\Bbbk$ in $n$ variables. A graded artinian $\Bbbk$-algebra $A=R/I$ is said to have the \textit{strong Lefschetz property}  if there is a linear form $\ell\in A_1$ such that the multiplication
$$\times \ell^s: A_{i}\longrightarrow A_{i+s}$$
has maximal rank for all $s$ and all $i$, i.e., $\times \ell^s$ is either injective or surjective, for all $s$ and all $i$. Such linear form $\ell$ is called a \textit{strong Lefschetz element} of $A.$

Three papers represent the beginning of the study of what is presently called \textit{Lefschetz properties} that were written by R. P. Stanley \cite{Stanley1980} in 1980; by J. Watanabe \cite{Watanabe1987} in 1987; and by L. Reid, L. G. Roberts and M. Roitman \cite{RRR1991} in 1991. These papers proved essentially the following same result.

\begin{Theoetoile}[Stanley's theorem]
If $\Bbbk $ is a field of characteristic zero and $R=\Bbbk[x_1,x_2,\ldots,x_n]$, then every artinian monomial complete intersection algebra
$$A=R/(x_1^{a_1},x_2^{a_2},\ldots,x_n^{a_n})$$
 has the strong Lefschetz property with $x_1+x_2+\cdots+x_n$ as a strong Lefschetz element.
\end{Theoetoile}
In the case $\Bbbk=\C$, Stanley used the fact that $A$ is isomorphic to the cohomology ring of a product of projective spaces and he applied the Hard Lefschetz Theorem to conclude that the divisor lattice of monomials in $A$ has the Sperner property. Watanabe proved the result by using the theory of modules over special linear Lie algebra $\mathfrak{sl}_2$, and Reid, Roberts and Roitman used Hilbert function techniques from commutative algebra in order to obtain this theorem.
The conclusion in these three papers as well as the techniques of proof are of interest in algebraic geometry, combinatorics, representation theory, and commutative algebra. 

In this paper, we will give a new proof of this theorem by  using only the basic properties of linear algebra. Furthermore, our proof is true not only for the case where the characteristic of $\Bbbk$ is zero, but also for the case where the characteristic is large enough, see Theorem~\ref{Theorem3.1}. More precisely, firstly let $I=(x_1^{2},x_2^{2},\ldots,x_n^{2})$ be an artinian quadratic monomial complete intersection ideal of $R$.  We see that the set of all square-free monomials of degree $i$ forms a $\Bbbk$-basis of  $\Bbbk$-vector space $B_i$, for all $i$, where $B=R/I$. Based on these bases, we construct the matrix of the multiplications  $\times (x_1+x_2+\cdots+x_n)^t: B_i\longrightarrow B_{i+t}$ for all $0\leq i\leq n$ and $0\leq t\leq n-i$ and show that these multiplications have maximal rank.  The main result is the following.

\begin{Theoetoile}[Theorem~\ref{Theorem3.7}]
Assume $\Bbbk$ is of characteristic zero or greater than $n$. Then the algebra $B=R/(x_1^2,x_2^2,\ldots,x_n^2)$ has the strong Lefschetz property  with $x_1+x_2+\cdots+x_n$ as a strong Lefschetz element.
\end{Theoetoile}
Then we show that any artinian monomial complete intersection $A$ can be viewed as a sub-algebra of an artinian quadratic monomial complete intersection $B$ such that $A$ and $B$ have the same socle degree. By applying Theorem~\ref{Theorem3.7}, we get that $A$ has also the strong Lefschetz property, namely Stanley's theorem is proved.
Furthermore,  our proof is still true in the case where the characteristic of $\Bbbk$ is greater than the socle degree of $A$, namely $a_1+a_2+\cdots+a_n - n$, see Theorem~\ref{Theorem3.1}.

\section{Artinian quadratic monomial complete intersection algebras}
\begin{Defi}
For any standard graded artinian  $\Bbbk$-algebra $A=R/I=\bigoplus_{i=0}^d A_i$, the \textit{Hilbert function} of $A$ is the function
$$h_A: \N\longrightarrow \N$$
defined by $h_A(t)=\dim_\Bbbk A_t$. As $A$ is artinian, its Hilbert function is equal to its \textit{$h$-vector} that one can express as a finite sequence
$$\underline{h}_A=\left(h_0,h_1,h_2,\ldots, h_d\right),$$
with $h_i=h_A(i)>0$ and $d$ is the last index with this property. The integer $d$ is called the \textit{socle degree} of $A$.
	
The $h$-vector  $\underline{h}_A$ is said to be \textit{unimodal} if there exists an integer $t\geq 1$ such that $h_0\leq h_1 \leq h_2\leq \cdots \leq h_t\geq h_{t+1}\geq \cdots\geq h_d.$ The $h$-vector  $\underline{h}_A$ is said to be \textit{symmetric} if $h_{d-i}=h_i$ for every $i=0,1,\ldots,\lfloor\frac{d}{2}\rfloor.$
\end{Defi}

In this section, we consider the case where $I=(x_1^2,x_2^2,\ldots,x_n^2)$ is an artinian quadratic monomial complete intersection ideal  in $R=\Bbbk[x_1,x_2,\ldots,x_n]$. Now set  $A=R/I=\bigoplus_{j= 0}^n A_j$. Hence $A=R/I$ is an artinian complete intersection of the socle degree $n$, namely $h_A(j)=0$  for all $j>n$ and moreover
$$ h_A(j)=h_A(n-j)=\binom{n}{j}$$
for all $j=0,1,\ldots, n$. The $h$-vector of $A$ is
$$\left(1,n, \binom{n}{2},\binom{n}{3},\ldots,\binom{n}{3}, \binom{n}{2},n,1\right).$$
In particular, the $h$-vector of $A$ is unimodal and symmetric.
Furthermore, we have the following.
\begin{Lem}\label{Lemma3.2}
The set of all square-free monomials forms a $\Bbbk$-basis of $A$.
\end{Lem} 

We denote by $\mathfrak{B}$ the set of all square-free monomials in $R$ and by $\mathfrak{B}_t$ the set of all square-free monomials of degree $t$ in $R.$ By Lemma~\ref{Lemma3.2}, $\mathfrak{B}$ is a $\Bbbk$-basis of $A$ and $\mathfrak{B}_t$ is a $\Bbbk$-basis of $A_t.$
We will list the monomials of $\mathfrak{B}_t$ in the decreasing order with respect to the reverse lexicographic order  with $x_1>x_2>\cdots >x_n$, i.e.,
$$\mathfrak{B}_t=\{x_{i_1}x_{i_2}\ldots x_{i_t}\mid 1\leq i_1<i_2<\cdots<i_t\leq n\}.$$
For example

\vskip 4mm
$ \mathfrak{B}_1=\{x_1,x_2,\ldots,x_n\},$

$
\mathfrak{B}_2=\{x_1x_2,x_1x_3,x_2x_3,\ldots,x_1x_n,x_2x_n,x_3x_n,\ldots,x_{n-1}x_n\} \text{ and }$

$
\mathfrak{B}_3=\{x_1x_2x_3,x_1x_2x_4,x_1x_3x_4,\ldots,x_{n-3}x_{n-2}x_{n-1},x_1x_2x_n,x_1x_3x_n,\ldots,x_{n-2}x_{n-1}x_n\}.$

\vskip 4mm

Denote by $M_{m\times n}(\Bbbk)$ the set of all $m\times n$ matrices with entries in the field $\Bbbk$. For any matrix $M\in M_{m\times n}(\Bbbk)$, it is known that $\rank(M)\leq \min\{m,n\}$. We say that $M$ has \textit{ maximal rank } if $\rank(M)=\min \{m,n\}$.

Now we fix a general linear form $\ell=a_1x_1+a_2x_2+\cdots+a_nx_n$ of $R.$ Consider the multiplication  $\times \ell^t: A_i\longrightarrow A_{i+t}$ for some $0\leq i\leq n$ and $0\leq t\leq n-i.$
Let $M_i^t$ be the matrix of $\times \ell^t$ with respect to the two bases $\mathfrak{B}_{i}$ and $\mathfrak{B}_{i+t}$. Thus $\times \ell^t:  A_{i}\longrightarrow A_{i+t}$ has maximal rank if and only if $M_i^t$ has maximal rank. When $t=1$  we will denote by $M_i$ instead of $M_i^1$. Note that when $t=0$,  $M_i^t$  is the identity matrix.
\begin{pro} \label{Proposition3.3}
With the above notations. Then the following assertions are equivalent:
\begin{enumerate}
	\item[\rm (i)] $A$ has the strong Lefschetz property.
	\item[\rm (ii)]  $M_i^t$ has maximal rank for all $0\leq i\leq n$ and $0\leq t\leq n-i$.
	\item[\rm (iii)]  $M_i^{n-2i}$ has maximal rank for all $0\leq i<\frac{n}{2}.$
\end{enumerate}
\end{pro}
\begin{proof}
Clearly, (i) is equivalent to (ii) by the definition. The equivalence of (ii) and (iii) follows from the basic properties of compositions of linear applications and the fact that $\dim_k (A_i)=\dim_k(A_{n-i})$ for all $0\leq i<\frac{n}{2}.$
\end{proof}
Now, set $\overline{R}:=R/(x_n)\cong \Bbbk[x_1,x_2,\ldots,x_{n-1}]$ and denote by $\overline{I}$ the image of $I$ in $\overline{R}.$ Therefore, $ \overline{I}=(x_1^2,x_2^2,\ldots,x_{n-1}^2)$ and  $$\overline{A}:=\overline{R}/\overline{I}=\frac{\Bbbk[x_1,x_2,\ldots,x_{n-1}]}{(x_1^2,x_2^2,\ldots,x_{n-1}^2)}$$
is also an artinian quadratic monomial complete intersection algebra. Denote by $\overline{\mathfrak{B}}_t$ the set of all square-free monomials of degree $t$ in $\overline{R}.$ By Lemma~\ref{Lemma3.2},
$$\overline{\mathfrak{B}}_t=\{x_{i_1}x_{i_2}\ldots x_{i_t}\mid 1\leq i_1<i_2<\cdots<i_t\leq n-1\}$$
is a $\Bbbk$-basis of $\overline{A}_t$.
Note  that
\begin{equation}\label{equation3.1}
\mathfrak{B}_t=\overline{\mathfrak{B}}_t \sqcup \mathfrak{B}_t^\prime,
\end{equation}
where
$$\mathfrak{B}_t^\prime=\{x_{i_1}x_{i_2}\ldots x_{i_{t-1}}x_n\mid 1\leq i_1<i_2<\cdots<i_{t-1}\leq n-1 \}.$$ 
We identify $\mathfrak{B}_t^\prime$ with the set
$$\overline{\mathfrak{B}}_{t-1}= \{x_{i_1}x_{i_2}\ldots x_{i_{t-1}}\mid 1\leq i_1<i_2<\cdots<i_{t-1}\leq n-1 \}.$$

Let $\ell=a_1x_1+a_2x_2+\cdots+a_nx_n$ be a linear form in $R$ and let $\overline{\ell}$ be the image of $\ell$ in $\overline{R}$. We denote by $\overline{M}_i^t$ the matrix of $\times \overline{\ell}^t:\overline{A}_{i}\longrightarrow \overline{A}_{i+t}$ with respect to the  two bases $\overline{\mathfrak{B}}_{i}$ and $\overline{\mathfrak{B}}_{i+t}$.
\begin{Lem} \label{Lemma3.4}
For any  $1\leq i\leq n-1$ and $1\leq t\leq n-i$, the matrix $M_i^t$ of $\times \ell^t: A_{i}\longrightarrow A_{i+t}$  with respect to the two bases $\mathfrak{B}_{i}$ and $\mathfrak{B}_{i+t}$ can be expressed by a $2\times 2$ block matrix of form
$$M_i^t=\bmt{ \overline{M}_i^t & & 0\\\\  a_nt\overline{M}_i^{t-1} &  & \overline{M}_{i-1}^t},$$
 where  $0$ is the zero matrix. 
\end{Lem}
\begin{proof}
We see immediately that
$$\ell^t= \sum_{j=0}^t \binom{t}{j} (\overline{\ell})^{t-j}(a_nx_n)^j =\overline{\ell}^t+  a_nt(\overline{\ell})^{t-1}x_n$$
in $A$.  The result follows from  the definition of the matrix  of $\times \ell^t: A_i\longrightarrow A_{i+t}$  with respect to the two bases $\mathfrak{B}_i$ and $\mathfrak{B}_{i+t}$ and using the decomposition of the two bases $\mathfrak{B}_{i}$ and $\mathfrak{B}_{i+t}$ as in \eqref{equation3.1} where we  identify
$$\mathfrak{B}_i^\prime\equiv \overline{\mathfrak{B}}_{i-1}\quad \text{and}\quad \mathfrak{B}_{i+t}^\prime\equiv \overline{\mathfrak{B}}_{i+t-1}.$$
We get the desired conclusion. 
\end{proof}
\begin{Exem}
Assume the characteristic of $\Bbbk$ is zero or greater than $4$. Consider $R=\Bbbk[x_1,x_2,x_3,x_4]$, $A=R/{(x_1^2,x_2^2,x_3^2,x_4^2)}$ and the linear form $\ell=x_1+x_2+x_3 +x_4$. Then the matrix of the multiplication $\times \ell^2: A_1\longrightarrow A_3$ with respect to the bases $\mathfrak{B}_1$ and $\mathfrak{B}_3$ is
\[
M_1^2=\left[\begin{array}{c c c|c}
2& 2 & 2 &0 \\\hline 2& 2 & 0 &2 \\2& 0 & 2 &2 \\0& 2 & 2 &2
\end{array}   \right] =\bmt{ \overline{M}_1^2 & & 0\\\\  2\overline{M}_1 &  & \overline{M}_{0}^2}
\]
A straightforward computation shows that $\det(M_1^2)=-2^4.3^2\neq 0 $, hence  the map $\times \ell^2: A_1\longrightarrow A_3$ is an isomorphism.
\end{Exem}

The following lemma is useful to determine the rank of block matrices.
\begin{Lem}\label{Lemma3.5}
Let $A\in M_{m\times n}(\Bbbk), B\in M_{n\times p}(\Bbbk), P\in M_{n\times n}(\Bbbk)$ such that $P$ is nonsingular. Assume that $M$ is a $(m+n)\times (n+p)$ matrix such that $M$ is written in $2\times 2$ block matrix as follows
$$M=\bmt{AP&& 0\\  \\ P &  & PB},$$
Then
$$\rank(M)= n+ \rank(APB).$$
\end{Lem}

\begin{proof}
We have that
$$\bmt{I_m&& -A\\  \\ 0 &  & P^{-1}} \times \bmt{AP&& 0\\  \\ P &  & PB}\bmt{I_n&& B\\  \\ 0 &  & -I_p}=\bmt{0&& APB\\  \\ I_n&  & 0},$$
and conclude immediately.
\end{proof}

Now we prove the main result in this section.
\begin{Theo}\label{Theorem3.7}
Assume that the characteristic of $\Bbbk$ is zero or greater than $n$. Then the algebra $A=R/(x_1^2,x_2^2,\ldots,x_n^2)$ has the strong Lefschetz property  with $\ell=x_1+x_2+\cdots+x_n$ as a strong Lefschetz element.
\end{Theo}
\begin{proof}
Set $\ell=x_1+x_2+\cdots+x_n$. Note first that since $\dim_\Bbbk(A_0)=\dim_\Bbbk(A_n)=1$, the matrix  of the map $\times\ell^n: A_{0}\longrightarrow A_{n}$ is $M_0^n=\pmt{n!}$. Therefore, $\det(M_0^n)=n!\ne 0$ since  the characteristic of $\Bbbk$ is zero or greater than $n$. 

Now to prove the theorem, by Proposition~\ref{Proposition3.3}, it is enough to show that the matrix $M_i^{n-2i}$  of $\times \ell^{n-2i}: A_{i}\longrightarrow A_{n-i}$ has maximal rank for all $0\leq i<\frac{n}{2}$. We show the assertion by induction on $n$.

Firstly, for $n=1,2$. We only have the case $i=0$. However, in this case the assertion holds as we have remarked at the beginning of the proof.
Now we assume that the assertion holds for any artinian quadratic monomial complete intersection  algebra in the polynomial ring in $<n$ variables.  For integers $n\ge 3$ and $i$ satisfying  $0< i<\frac{n}{2}$, we have to show that  the matrix $M_i^{n-2i}$  of $\times \ell^{n-2i}: A_{i}\longrightarrow A_{n-i}$ has maximal rank. 
By Lemma~\ref{Lemma3.4}, $M_i^{n-2i}$ can be written in the form of $2\times 2$ block matrix
$$M_i^{n-2i}=\bmt{ \overline{M}_i^{n-2i} & & 0\\\\ (n-2i) \overline{M}_i^{n-2i-1} &  & \overline{M}_{i-1}^{n-2i}}.$$
 Note that $0<n-2i<n$, hence $n-2i$ is an invertible element of $\Bbbk$. Observe that   $M_i^{n-2i}$ and $\overline{M}_i^{n-2i-1}$  are  square matrices of size $\binom{n}{i} \times \binom{n}{i}$ and $\binom{n-1}{i} \times \binom{n-1}{i}$, respectively. Moreover,
$$\overline{M}_i^{n-2i}=\overline{M}_{n-i-1} \overline{M}_i^{n-2i-1}\quad \text{and}\quad \overline{M}_{i-1}^{n-2i}= \overline{M}_i^{n-2i-1} \overline{M}_{i-1}.$$

As $\overline{A}=\overline{R}/\overline{I}$ has the strong Lefschetz property by the inductive hypothesis, $\overline{M}_i^{n-2i-1}$ is a nonsingular matrix. It follows from Lemma~\ref{Lemma3.5} that
$$\rank(M_i^{n-2i})=\binom{n-1}{i} +\rank(\overline{M}_{n-i-1} \overline{M}_i^{n-2i-1}\overline{M}_{i-1}   ).$$

We observe that $\overline{M}_{n-i-1} \overline{M}_i^{n-2i-1}\overline{M}_{i-1}$ is the matrix of $$\times(\overline{\ell})^{n-2i+1}: \overline{A}_{i-1}\longrightarrow \overline{A}_{n-i}.$$
It is an isomorphism since  $\overline{A}$ has the strong Lefschetz property. Hence
$$\rank(\overline{M}_{n-i-1} \overline{M}_i^{n-2i-1}\overline{M}_{i-1})=\binom{n-1}{i-1} .$$
It follows that 
$$\rank(M_i^{n-2i})=\binom{n-1}{i} +\binom{n-1}{i-1}=\binom{n}{i}.$$
This implies that $\ell^{n-2i}: A_i\longrightarrow A_{n-i}$ is an isomorphism. This finishes the induction and the proof.
\end{proof}
\section{Proof of Stanley' theorem}
Finally, we show that Stanley's theorem follows from Theorem~\ref{Theorem3.7}.
\begin{Theo}\label{Theorem3.1}
Let $\Bbbk$ be a field and $R=\Bbbk[x_1,x_2,\ldots,x_n]$. Then  $$R/(x_1^{d_1},x_2^{d_2},\ldots,x_n^{d_n})$$
has the strong Lefschetz property if the characteristic of $\Bbbk$ is zero or greater than $d_1+d_2+\cdots+d_n-n$.
\end{Theo}
\begin{proof}
To simplicity, denote $A=R/(x_1^{a_1+1},x_2^{a_2+1},\ldots,x_n^{a_n+1})$, with $a_1,a_2,\ldots,a_n$ are the positive integers.  Assume that the characteristic of $\Bbbk$ is zero or greater than $a_1+a_2+\cdots+a_n$. We need to show that $A$ has the strong Lefschetz property.

Note that $A$ is an artinian monomial complete intersection algebra with the socle degree $m=a_1+a_2+\cdots+a_n$. Set $\alpha_i=\sum_{j=1}^i a_j$ for $i=1,\ldots,n$. Now we let $B~=~\Bbbk[x_1,x_2,\ldots,x_m]/(x_1^2,x_2^2,\ldots,x_m^2)$. By Theorem~\ref{Theorem3.7}, the algebra $B$ has the strong Lefschetz property with $\ell=x_1+x_2+\cdots+ x_m$ as a strong Lefschetz element.  Set $S:=\Bbbk[y_1,y_2,\ldots,y_n]$ and consider the algebra homomorphism $\phi: S\longrightarrow B$ given by
\begin{align*}
y_1&\longmapsto x_1+\cdots+x_{\alpha_1}\\
y_2&\longmapsto x_{\alpha_1+1}+\cdots+x_{\alpha_2}\\
&\cdots\\
y_n&\longmapsto x_{\alpha_{n-1}+1}+\cdots+x_{\alpha_n}.
\end{align*}
Set $J=(y_1^{a_1+1},y_2^{a_2+1},\ldots,y_n^{a_n+1})$. We have the following assertion.

\noindent\textsc{Claim:} $\Ker(\phi)=J$. 

\noindent{\it Proof of Claim:} First, for each $j=1,2,\ldots,n$, we see that $$\phi(y_j^{a_j+1})=(\underbrace{x_{\alpha_{j-1}+1}+\cdots+x_{\alpha_j}}_{a_j\;\text{variables }})^{a_j+1}=0$$
in $B$.  In other words, $y_j^{a_j+1}\in \Ker(\phi)$, so $J\subset \Ker(\phi)$. Assume the contrary that $\Ker(\phi)\ne J$. Select a homogeneous element $f$ of largest degree with $f\in \Ker(\phi)$ and $f\notin J$. It follows that $f$ represents a non-trivial element in the socle of $S/J$. Note that $S/J$ is an artinian monomial complete intersection algebra and its socle is a $\Bbbk$-vector space  spanned by $y_1^{a_1}\ldots y_n^{a_n}$. Hence there exists a non-zero element $ c\in \Bbbk$ such that $f=cy_1^{a_1}\ldots y_n^{a_n}+g$, where $g\in J$. Thus, $\phi(g)=0$ and
$$\phi(f)= c(x_1+\cdots+x_{\alpha_1})^{a_1}\ldots (x_{\alpha_{n-1}+1}+\cdots+x_{\alpha_n})^{a_n}=ca_1!\ldots a_n! x_1x_2\ldots x_m.$$
Note that the characteristic of $\Bbbk$ is zero or greater than $m=a_1+a_2+\cdots+a_n$, hence $ca_1!\ldots a_n!$ is an invertible element of $\Bbbk$. Thus $\phi(f)=ca_1!\ldots a_n! x_1x_2\ldots x_m \ne 0$ in $B$. This  contradicts  $f\in \Ker(\phi).$ The contradiction completes the proof of Claim.

By the above claim, we get the  following algebra isomorphisms
\[
A\simeq S/\Ker(\phi)\simeq \Image(\phi) \subseteq B.
\]
It follows that we can identify  $A$ with a sub-algebra of $B$. Furthermore, $A$ and $B$ have the same socle degree, namely $m$.  We have the commutative diagrams
\begin{align}\label{Commutativediagram}
\xymatrix{ A_0\ar[r]^{\times \ell}\ar@{_{(}->}[d] &A_1   \ar[r]^{\times \ell} \ar@{_{(}->}[d]&A_2 \ar[r]^{\times \ell} \ar@{_{(}->}[d]&\cdots \ar[r]^{\times \ell}\ar@{_{(}->}[d] &   A_{m-2}\ar[r]^{\times \ell}\ar@{_{(}->}[d]&A_{m-1}\ar[r]^{\times \ell}\ar@{_{(}->}[d]&A_m \ar@{_{(}->}[d]  \\
B_0\ar[r]^{\times \ell} &B_1   \ar[r]^{\times \ell} &B_2 \ar[r]^{\times \ell} &\cdots \ar[r]^{\times \ell} &   B_{m-2}\ar[r]^{\times \ell}&B_{m-1}\ar[r]^{\times \ell}&B_m,  }
\end{align} 
where the vertical maps are the canonical inclusions and $\ell=x_1+x_2+\cdots+ x_m$. Since $B$ has the strong Lefschetz property,  the maps $\times \ell^{m-2i}: B_i\longrightarrow B_{m-i}$ is bijective, for all $0\le i<\frac{m}{2}$. By \eqref{Commutativediagram}, we get that  $\times \ell^{m-2i}: A_i\longrightarrow A_{m-i}$ is injective. Since $\dim_\Bbbk A_i = \dim_\Bbbk A_{m-i}$, it implies that $\times \ell^{m-2i}: A_i\longrightarrow A_{m-i}$ is bijective, for all $0\le i<\frac{m}{2}$. Thus $A$ has the strong Lefschetz property. The proof is completed.
\end{proof}

\section*{Acknowledgments}
This work  is  supported by the Vietnam Ministry of Education and Training under grant number B2022-DHH-01. The second author is also partially supported by the Core Research Program of Hue University under grant number NCM.DHH.2020.15. Part of this paper was written while the second author visited the Vietnam Institute for Advanced Study in Mathematics (VIASM), he would like to thank VIASM for the very kind support and hospitality. The authors are grateful to the anonymous referee for his/her careful reading of the manuscript and many useful suggestions.


\end{document}